\documentclass[onecolumn,10pt]{article}
\usepackage[top=.75in, bottom=.75in, left=.75in, right=.75in]{geometry}
\usepackage{amsmath,amsfonts,amscd,amssymb}
\usepackage{bbm}
\usepackage{amsthm}
\usepackage{graphicx}
\usepackage{epstopdf}
\usepackage{overpic}
\usepackage{cancel}
\usepackage{rotating}
\usepackage{url}
\usepackage{caption}
\usepackage{color}
\usepackage{rotating}
\usepackage{multirow}
\usepackage{wrapfig}
\usepackage{mathtools}
\usepackage{subeqnarray}
\usepackage{setspace}
\usepackage{palatino} 
\usepackage[numbers,sort&compress]{natbib}

\usepackage[bottom,flushmargin,hang,multiple]{footmisc}
\usepackage{lipsum}
\newcommand\blfootnote[1]{%
  \begingroup
  \renewcommand\thefootnote{}\footnote{#1}%
  \addtocounter{footnote}{-1}%
  \endgroup
}

\numberwithin{equation}{section}
\numberwithin{table}{section}
\numberwithin{equation}{section}
\newtheorem{theorem}{Theorem}[section]

\newtheorem{corollary}[theorem]{Corollary}
\newtheorem{definition}[theorem]{Definition}
\newtheorem{example}[theorem]{Example}

\newtheorem{remark}[theorem]{Remark}

\definecolor{header1}{cmyk}{0,0,0,1}

\usepackage[utf8]{inputenc}

\usepackage[normalem]{ulem}
\usepackage{color}

\title{\vspace{-.125in}{\huge\selectfont \textbf{Deriving Newton's Canonical Forms of Cubic Curves via the Center of Polynomials}}\vspace{-.075in}}

\author{\normalsize{Hua-Lin Huang$^{1}$, Lili Liao$^{2,*}$}, Yu Ye$^{3, 4}$ and Ziqi Yuan$^{3}$\\
\footnotesize{$^1$ School of Mathematical Sciences, Huaqiao University,  Quanzhou 362021, China}\\
\footnotesize{$^2$ School of Mathematical Sciences, Anhui University, Hefei 230601, China }\\
\footnotesize{$^3$ School of Mathematical Sciences, University of Science and Technology of China, Hefei 230026, China} \\
\footnotesize{$^4$ Hefei National Laboratory, University of Science and Technology of China, Hefei 230026, China \vspace{-.2in}}}
\date{}

\begin{document}
\maketitle

\blfootnote{$^*$ Corresponding author: lililiao@stu.ahu.edu.cn}
\vspace{-.2in}
\begin{abstract}
Existing classifications of real plane cubic curves rely on sophisticated tools and require a lengthy exposition. This paper provides a concise yet elementary treatment of this classical topic. Using the centers of polynomials, we establish a correspondence between the algebraic structure of these centers and geometric properties of binary cubic polynomials, which yields a novel approach to Newton’s canonical forms.
\end{abstract}

\section{Introduction}\label{Sec:Intro}
Since Descartes and Fermat founded analytic geometry, it has been recognized that conic curves are essentially real binary quadratic polynomials, which led to the expectation of using this model to explore higher degree curves. Newton first initiated the systematic study of cubic curves using analytic geometry, and achieved foundational results in his celebrated work \emph{Enumeratio Curvarum Trium Dimensionum}~\cite{N04}. The core of this work lies in simplifying the general binary cubic polynomial with ten coefficients through coordinate changes such as rotation, translation, and scaling, thereby obtaining four canonical forms of cubic curves
\begin{align}  
\label{5.1} xy^{2} + ey &= ax^{3} + bx^{2} + cx + d, \\  
\label{5.2} xy & = ax^{3} + bx^{2} + cx + d, \\  
\label{5.3} y^{2} & = ax^{3} + bx^{2} + cx + d, \\  
\label{5.4} y & =  ax^{3} + bx^{2} + cx + d.  
\end{align}  
Furthermore, Newton claimed that any irreducible cubic curve can be transformed into the form of (1.3). In terms of geometry, all irreducible cubics could be obtained from those in (1.3) by projecting between planes. 

This masterpiece by Newton demonstrated ingenious geometric intuition and exceptional computational skill. For example, in his first attempt Newton handled intermediate equations containing 84 terms by selecting appropriate coordinate systems based on asymptotes and inflection points~\cite{W}, to name a few. Inspired by his work, many renowned mathematicians, such as Cayley~\cite{C} and Plücker~\cite{P}, restudied cubic curves from various perspectives and ended up with the same classification. Klein’s Erlangen Program provided a classification criterion based on group actions, with related work documented in~\cite{BH, K}. Overall, the classification of plane cubic curves is highly non-trivial. Even the most modern treatments still require considerable length~\cite{B, BM, BK} and rely on relatively complex knowledge such as singularities, inflection points, and Bézout's theorem. Therefore, exploring a more concise and comprehensible approach is both necessary and valuable.

This paper aims to derive Newton's canonical forms for cubic curves via a concise yet elementary approach. From an algebraic perspective, the classification of curves and surfaces of degree two, three, or higher amounts to eliminating as many cross terms and redundant terms as possible from a general polynomial by affine changes of variables, so as to obtain a simplified canonical form. The authors observe that the center of polynomials developed in~\cite{FHL} can be an effective tool for addressing this problem. In brief, a multivariate polynomial \( f(x_1, x_2, \dots, x_n)\) can be expressed as the sum of two polynomials in disjoint sets of variables \( g(x_1, \dots, x_a) + h(x_{a+1}, \dots, x_n) \) if and only if its Hessian matrix \( H_f = (\frac{\partial^2 f}{\partial{x_i}\partial{x_j}}) \) is block diagonal. Furthermore, \( f \) can be expressed as a sum of polynomials in disjoint variables through a change of variables if and only if its Hessian matrix is congruent to a block diagonal matrix. Therefore, we defined an invariant algebra up to changes of variables for \( f \), denoted as \[ Z(f) := \{ X \in \mathbbm{R}^{n \times n} \mid (H_f X)^T = H_f X \} \] and called the center of \( f \). This center is used to determine and compute issues such as separation of variables and simplification for \( f \). We find that this invariant algebra is a perfect tool for the purpose of simplifying and classifying binary cubic polynomials. 

More precisely, by considering simultaneously the center \(Z(f)\) of the whole polynomial and that \(Z(f_3)\) of its cubic homogeneous component, we can reduce all binary cubic polynomials into nine simple forms. In particular, Newton's last three canonical forms (1.2)–(1.4) have degenerate cubic components, or equivalently $\dim Z(f_3)=3$, and are entirely determined by \(Z(f)\) and \(Z(f_3)\). A binary cubic polynomial \(f\) can be reduced to Newton's first canonical form (1.1) if and only if $\dim Z(f_3)=2$, namely the cubic component is nondegenerate. This is the most difficult type which contains the general binary cubic polynomials, and can be subdivided into six types according to the algebraic structures of the pair of centers $Z(f)$ and $Z(f_3)$. In other words, we obtain a more refined classification of cubic curves using only linear algebraic techniques. Furthermore, using the fact that Newton's canonical form (1.3) is completely determined by the condition $\mathbbm{R}^2 \cong Z(f) \subset Z(f_3) \cong T$ where $T$ is the algebra of $2 \times 2$ lower-triangular matrices, we provide an alternative proof for Newton's claim that every irreducible cubic curve can be transformed into the form of (1.3). Our method falls within the scope of invariant theory but breaks through the limitations of classical invariant theory, which often considers only invariant polynomials. Instead, at a broader level we consider invariant algebras as a classification criterion, thereby achieving the classification goal with basic algebraic tools and in a more concise manner. This approach is simple and clear, and it holds potential for application to more general curves, hypersurfaces, and even algebraic varieties in the future.

The remainder of this paper is organized as follows: in Section 2 we review the centers of polynomials, and derive the canonical forms of real binary cubics via their centers; in Section 3 we simplify and classify binary cubic polynomials by the algebraic structures of centers, and derive Newton's canonical forms of cubic curves. 
\section{Preliminaries}

This section reviews the center of multivariate polynomials and its application to direct sum decomposition. Using this approach, we revisit the canonical forms of real binary cubics. For convenience, the base field in this paper is restricted to the real numbers \(\mathbbm{R}\). However, the relevant theories, methods, and results can be adapted to arbitrary fields with appropriate modifications.

\subsection{The Center of a Polynomial}

The center of a homogeneous multivariate polynomial arises naturally in the study of many algebraic and geometric problems, see, for example, \cite{CG, H}. To address related issues for non-homogeneous multivariate polynomials, we extended the definition of centers to a more general setting \cite{FHL}.

\begin{definition}
Let \( f \in \mathbbm{R}[x_1, x_2, \dots, x_n] \) be an \( n \)-variate polynomial, and let \( H_f \) be its Hessian matrix. The center of the polynomial \( f \) is defined as  
		\begin{equation}\label{center}
			Z(f) := \bigl\{ X \in \mathbbm{R}^{n \times n} \mid (H_f X)^{T} = H_f X \bigr\}.
	\end{equation}
\end{definition}

Clearly, \( Z(f) \) is a nonzero subspace of \( \mathbbm{R}^{n \times n} \), which contains at least all the scalar matrices.

\begin{example}\label{23}
{\upshape Consider the real ternary cubic polynomial  
\[
\begin{aligned}
f(x_{1},x_{2},x_{3}) = & \; 375x_1^3 - 225x_1^2x_2 - 225x_1^2x_3 + 45x_1x_2^2 +90x_1x_2x_3+ 45x_1x_3^2-3x_2^3-9x_{2}^{2}x_{3}-9x_{2}x_{3}^{2}-3x_3^3\\
                       &+ 50x_1^2 + 6x_2^2 + 12x_2x_3 + 13x_3^2-20x_1x_2 - 20x_1x_3 
                       + 25x_{1} - x_{2} + 4x_{3} + 7.
\end{aligned}
\]

The Hessian matrix of \( f(x_{1},x_{2},x_{3}) \) is computed as  
\[
H_{f} = 2
\begin{pmatrix}
1125x_{1}-225x_{2}-225x_{3}+50 & -225x_{1}+45x_{2}+45x_{3}-10  & -225x_{1}+45x_{2}+45x_{3}-10\\
			-225x_{1}+45x_{2}+45x_{3}-10 & 45x_{1}-9x_{2}-9x_{3}+6 & 45x_{1}-9x_{2}-9x_{3}+6\\
            -225x_{1}+45x_{2}+45x_{3}-10 & 45x_{1}-9x_{2}-9x_{3}+6 & 45x_{1}-9x_{2}-9x_{3}+13
\end{pmatrix}.
\]

Using Equation \eqref{center}, the center of \( f \) is obtained as  
\[
Z(f) = \left\{ 
\begin{pmatrix}
a & b & b+7c \\
0 & a+5b-20c & d \\
0 & 20c & a+5b + 35c-d
\end{pmatrix}
\;\Bigg|\;
a,b,c,d \in \mathbbm{R}
\right\}.
\]}
\end{example}

It can be observed that the set \( Z(f) \) defined above is not closed under matrix multiplication, let alone possessing the structure of a commutative associative algebra which is a critical property of the center of a nondegenerate homogeneous polynomial. In fact, the center of a non‑homogeneous polynomial exhibits a special Jordan algebraic structure. Since the study of Jordan algebra structures is not within the scope of this paper, we will not elaborate on its details here. Interested readers may refer to the relevant literature, such as \cite{J68}.

\subsection{Centers and Direct Sum Decompositions of Polynomials}
We first recall several definitions. A polynomial \(f(x) = f(x_1, x_2, \dots, x_n) \in \mathbbm{R}[x_1, x_2, \dots, x_n]\) is called decomposable as a direct sum if there exists an affine change of variables \(x = Py + v\) such that  
\[
g(y) = g(y_1, y_2, \dots, y_n) := f(Py + v) = g_1(y_1, \dots, y_a) + g_2(y_{a+1}, \dots, y_n),
\]  
where \(P \in \operatorname{GL}(n, \mathbbm{R})\), \(v \in \mathbbm{R}^n\), and \(1 \le a \le n-1\). In particular, if \(g(y) = \sum_{i=1}^n g_i(y_i)\), that is, \(g\) can be expressed as a sum of univariate polynomials, then \(f\) is called diagonalizable.

The following theorem describes the relationship between the center of a general polynomial and its direct sum decompositions. For detailed proofs, we refer the reader to \cite{FHL}.

\begin{theorem}\label{center decompose}
Keep the previous notations.
\begin{itemize}
    \item[{(1)}] \( Z(g) = P^{-1}Z(f)P := \{P^{-1}XP \mid X \in Z(f)\} \).
    \item[{(2)}] There is a one-to-one correspondence between direct sum decompositions of \(f\) and decompositions of the identity matrix as a sum of mutually orthogonal idempotent matrices in its center \(Z(f)\).
    \item[{(3)}] \( f \) is indecomposable if and only if \( Z(f) \) contains no non-trivial idempotents.
    \item[{(4)}] \( f \) is diagonalizable if and only if the identity matrix can be expressed as a sum of \( n \) mutually orthogonal idempotents in its center \( Z(f) \).
\end{itemize}
\end{theorem}

In what follows, we elucidate the above approach by a concrete example.

\begin{example}
{\upshape (Continuation of Example \ref{23})  
We continue with the polynomial \( f(x_{1},x_{2},x_{3}) \) of Example \ref{23}. We may select from the computed center \( Z(f) \) a basis as follows
\[
\begin{aligned}
X_{1} &= \begin{pmatrix}
1 & 0 & 0\\
0 & 1 & 0\\
0 & 0 & 1
\end{pmatrix}, &
X_{2} &= \begin{pmatrix}
0 & 1 & 1\\
0 & 5 & 0\\
0 & 0 & 5
\end{pmatrix}, &
X_{3} &= \begin{pmatrix}
0 & 0 & 7\\
0 & -20 & 0\\
0 & 20 & 35
\end{pmatrix}, & 
X_{4}&= \begin{pmatrix}
0 & 0 & 0\\
0 & 0 & 1\\
0 & 0 & -1
\end{pmatrix}.
\end{aligned}
\]

Based on the algebraic structure of \( Z(f) \), we find three mutually orthogonal primitive idempotent matrices
\[
\epsilon_{1} = \begin{pmatrix}
1 & -\frac{1}{5} & -\frac{1}{5}\\[2pt]
0 & 0 & 0\\
0 & 0 & 0
\end{pmatrix},\quad
\epsilon_{2} = \begin{pmatrix}
0 & \frac{1}{5} & \frac{1}{5}\\[2pt]
0 & 1 & 1\\
0 & 0 & 0
\end{pmatrix},\quad
\epsilon_{3} = \begin{pmatrix}
0 & 0 & 0\\
0 & 0 & -1\\
0 & 0 & 1
\end{pmatrix}
\]
which add up to the identity matrix.

By direct computation, we obtain an invertible matrix  
\[
P = \begin{pmatrix}
\frac{1}{5} & \frac{1}{5} & 0\\
0 & 1 & -1\\
0 & 0 & 1
\end{pmatrix}
\]
such that
\[
P^{-1}\epsilon_{1}P = \begin{pmatrix}
1 & 0 & 0\\
0 & 0 & 0\\
0 & 0 & 0
\end{pmatrix},\quad
P^{-1}\epsilon_{2}P = \begin{pmatrix}
0 & 0 & 0\\
0 & 1 & 0\\
0 & 0 & 0
\end{pmatrix},\quad
P^{-1}\epsilon_{3}P = \begin{pmatrix}
0 & 0 & 0\\
0 & 0 & 0\\
0 & 0 & 1
\end{pmatrix}.
\]

Taking the following change of variables
\[
x_{1} = \frac{1}{5}(y_{1} + y_{2}),\quad x_{2} = y_{2} - y_{3},\quad x_{3} = y_{3},
\]  
we diagonalize \( f(x_{1},x_{2},x_{3}) \) into
\[
f = 3y_{1}^{3} + 2y_{1}^{2} + 5y_{1}+4y_{2}^{2}+ 4y_{2}+ 7y_{3}^{2} + 5y_{3} + 7.
\]}
\end{example}

\subsection{The Canonical Forms of Real Binary Cubics}

In this subsection, we employ the center to derive the canonical forms of real binary cubic forms.

Without loss of generality, a nonzero real binary cubic form can be written as  
\[
f(x,y) = Ax^{3} + 3Bx^{2}y + 3Cxy^{2} + Dy^{3},
\]  
where \(A,B,C,D\) are not all zero. For the convenience of subsequent applications, we define  
\[
D_{1}=\begin{vmatrix}
B & C \\[2pt]
C & D
\end{vmatrix},\quad
D_{2}=\begin{vmatrix}
C & A \\[2pt]
D & B
\end{vmatrix},\quad
D_{3}=\begin{vmatrix}
A & B \\[2pt]
B & C
\end{vmatrix}.
\]  
Then the Hessian determinant of \(f\) is 
$h_f=36(D_3 x^2 - D_2 xy + D_1 y^2).$
Further, let \(\Delta_{1} = D_2^{2} - 4D_1D_3\), and denote by  
\[
T = \left\{ 
\begin{pmatrix}
s & 0 \\
r & t
\end{pmatrix}
\;\Big|\; r,s,t \in \mathbbm{R} \right\}
\]  
the algebra of \(2 \times 2\) lower-triangular matrices.

Two binary cubic forms \(f\) and \(g\) are said to be equivalent, denoted as \(f \sim g\), if they can be transformed into each other by an invertible linear change of variables. Suppose  
\[
g(x,y) = f(ax+by,\ cx+dy) = A'x^{3} + 3B'x^{2}y + 3C'xy^{2} + D'y^{3}
\]  
and define  
\[
D'_{1}=\begin{vmatrix}
B' & C' \\[2pt]
C' & D'
\end{vmatrix},\quad
D'_{2}=\begin{vmatrix}
C' & A' \\[2pt]
D' & B'
\end{vmatrix},\quad
D'_{3}=\begin{vmatrix}
A' & B' \\[2pt]
B' & C'
\end{vmatrix}.
\]  
Since
\[ h_g(x, y)= (ad-bc)^2 h_f(ax+by, cx+dy),\]
when \(ad-bc\ne 0\), there holds 
\begin{equation}
D_1=D_2=D_3=0 \iff h_f=0 \iff h_g=0 \iff D_1'=D_2'=D_3'=0.\label{invD}
\end{equation}
Moreover, a direct computation yields 
\[\Delta_1'=(ad-bc)^6 \Delta_1.\]
Therefore, \(\Delta_1\) is an invariant of weight \(6\), see e.g. \cite{G, Olver}.

\begin{theorem}\label{cubicform}
    Keep the previous notations. For any nonzero binary cubic form, exactly one of the following holds.
    \begin{itemize}
        \item[(1)] \( f \sim x^{3} \) if and only if \( Z(f) \cong T \), if and only if \( D_1 = D_2 = D_3 = 0 \).
        \item[(2)] \( f \sim x^{3} + y^{3} \sim x^3 + 3xy^{2} \) if and only if \( Z(f) \cong \mathbbm{R}^{2} \), if and only if \( \Delta_{1} > 0 \).
        \item[(3)] \( f \sim x^{2}y \) if and only if \( Z(f) \cong \mathbbm{R}[\epsilon] / \langle \epsilon^{2} \rangle \), if and only if \( \Delta_{1} = 0 \) and \( (D_1,\ D_2,\ D_3)\ne 0\).
        \item[(4)] \( f \sim x^{3} - 3xy^{2} \) if and only if \( Z(f) \cong \mathbbm{C} \), if and only if \( \Delta_{1} < 0 \).
    \end{itemize}
\end{theorem}

\begin{proof}
    First, using Theorem \ref{center decompose} and \eqref{invD}, the``\(\Rightarrow\)" in each item of the theorem can be easily verified by computation.

   Conversely, to prove the “\(\Leftarrow\)”, we begin by computing the center of a general real binary cubic form \( f \). Its Hessian matrix is
\[
H_{f} = \begin{pmatrix}
    6Ax + 6By & 6Bx + 6Cy \\
    6Bx + 6Cy & 6Cx + 6Dy
\end{pmatrix}.
\]
By Equation \eqref{center}, the center \( Z(f) \) is the solution space of the linear system
\begin{align} \label{H3}
    \begin{cases}
        Ax_{12} \;+\; B(x_{22} - x_{11}) \;-\; Cx_{21} = 0, \\
        Bx_{12} \;+\; C(x_{22} - x_{11}) \;-\; Dx_{21} = 0.
    \end{cases}
\end{align}
Since \( A, B, C, D \) are not all zero, we have \( \dim Z(f) = 3 \) or \( 2 \).

Clearly, \( \dim Z(f) = 3 \) if and only if \( D_1 = D_2 = D_3 = 0 \). In other words, \( (A, B, C) \) and \( (B, C, D) \) are linearly dependent. Without loss of generality, assume that \( A\ne 0 \) and\( (B, C, D) = \lambda (A, B, C) \) if this is the case. Then we easily obtain \( f = A(x + \lambda y)^3 \), i.e., \( f \sim x^3 \). This proves item (1) of the theorem.

Next, we consider the case \( \dim Z(f) = 2 \), which is equivalent to \( (A, B, C) \) and \( (B, C, D) \) being linearly independent. Using the vector product of analytic geometry, the solution of system \eqref{H3} can be written as
\[
(x_{12},\; x_{22} - x_{11},\; -x_{21}) = b\, (A, B, C) \times (B, C, D)
= b\, (D_1, D_2, D_3), \quad \forall b \in \mathbbm{R}.
\]
In matrix form, this becomes
\[
\begin{pmatrix}
    x_{11} & x_{12} \\
    x_{21} & x_{22}
\end{pmatrix}
= a \begin{pmatrix}
    1 & 0 \\
    0 & 1
\end{pmatrix}
+ b \begin{pmatrix}
    0 & D_1 \\
    -D_3 & D_2
\end{pmatrix}, \quad \forall a, b \in \mathbbm{R}.
\]
Hence, as an algebra, \( Z(f) \) is generated by the matrix
\[
\Lambda = \begin{pmatrix}
    0 & D_1 \\
    -D_3 & D_2
\end{pmatrix}.
\]
The characteristic polynomial of \( \Lambda \) is \( x^{2} - D_{2}x + D_{1}D_{3} \).

First, suppose \( \Delta_{1} = D_{2}^{2} - 4D_{1}D_{3} > 0 \). 
Then \( x^{2} - D_{2}x + D_{1}D_{3} = 0 \) has two distinct real roots \( \alpha_{1} \) and \( \alpha_{2} \), so \( x^{2} - D_{2}x + D_{1}D_{3} = (x - \alpha_{1})(x - \alpha_{2}) \). By the Chinese remainder theorem,
\[
Z(f) \cong \mathbbm{R}[x] / \langle (x - \alpha_{1})(x - \alpha_{2}) \rangle \cong \mathbbm{R}^{2}.
\]
Now there exists an invertible matrix \( P \in \operatorname{GL}(2, \mathbbm{R}) \) such that
\[
P^{-1}\Lambda P = \begin{pmatrix}
    \alpha_1 & 0 \\
    0 & \alpha_2
\end{pmatrix},
\]
and consequently
\[
P^{-1}Z(f)P = \left\{ \begin{pmatrix}
    s & 0 \\
    0 & t
\end{pmatrix} \;\Big|\; s,t \in \mathbbm{R} \right\}.
\]
By Theorem \ref{center decompose}, after a change of variables \( (x, y) \mapsto (x, y) P^{T} \), we may assume that
\[
\begin{pmatrix} 1 & 0 \\ 0 & 0 \end{pmatrix},\;
\begin{pmatrix} 0 & 0 \\ 0 & 1 \end{pmatrix} \in Z(f).
\]
Substituting these into system \eqref{H3} yields \( B = C = 0 \), so \( f = Ax^{3} + Dy^{3} \). Hence \( f(x,y) \sim x^{3} + y^{3} \). Moreover, replacing \( x, y \) with \( \frac{x+y}{2} \) and \( \frac{x-y}{2} \) gives \( x^3 + y^3 \sim x^3 + 3xy^2 \). This proves item (2).

Second, assume \( \Delta_{1} = 0 \).
Then \( x^{2} - D_{2}x + D_{1}D_{3} = 0 \) has a double root \( \alpha \), i.e., \( x^{2} - D_{2}x + D_{1}D_{3} = (x - \alpha)^{2} \). Thus
\[
Z(f) \cong \mathbbm{R}[x] / \langle (x - \alpha)^{2} \rangle \cong \mathbbm{R}[\epsilon] / \langle \epsilon^{2} \rangle.
\]
By a suitable change of variables, we may assume \( \begin{pmatrix} 0 & 1 \\ 0 & 0 \end{pmatrix} \in Z(f) \). Substituting this into the definition of the center gives \( A = B = 0 \), so
\[
f = 3Cxy^{2} + Dy^{3} = (3Cx + Dy)y^{2},
\]
i.e., \( f \sim xy^{2} \). This proves item (3).

Finally, let \( \Delta_{1} < 0 \).  
In this case, the polynomial \( x^{2} - D_{2}x + D_{1}D_{3} = 0 \) has no real roots. Therefore,
\[
Z(f) \cong \mathbbm{R}[x] / \langle x^{2} - D_{2}x + D_{1}D_{3} \rangle \cong \mathbbm{C}.
\]
After an appropriate change of variables, we may assume that
\[
\begin{pmatrix} 0 & 1 \\ -1 & 0 \end{pmatrix} \in Z(f).
\]
Substituting this into system \eqref{H3} yields \( C = -A \) and \( D = -B \). Consequently,
\[
f = Ax^{3} + 3Bx^{2}y - 3Axy^{2} - By^{3}
       = \frac{1}{2}(A + Bi)(x - yi)^{3} + \frac{1}{2}(A - Bi)(x + yi)^{3}.
\]
Choose a complex number \( U + iV \) such that \( (U + iV)^{3} = A + Bi \). Setting \( p = Ux + Vy \) and \( q = Vx - Uy \) transforms the above into
\[
g(p,q) := \frac{1}{2}(p + iq)^{3} + \frac{1}{2}(p - iq)^{3}
        = p^{3} - 3pq^{2}.
\]
Thus \( f \sim x^{3} - 3xy^{2} \), which completes the proof of item (4).

The whole theorem is now proved.
\end{proof}

\begin{remark}\label{Re 1}
The method we employed to derive the canonical forms of real binary cubic forms differs significantly from the traditional approaches \cite{G, Olver}. Moreover, by using the center of polynomials, one can derive the Cardano formula of cubic equations by completing the cube \cite{HRXY25}. The theorem above also shows that the algebraic structure of the center of a binary cubic determines the number of its real roots and multiplicities, and vice versa.
\end{remark}

\section{Deriving Newton's Canonical Forms of Cubic Curves via Centers}

By taking the centers of both binary cubic polynomials and their cubic components into consideration at the same time, we reduce all binary cubic polynomials into nine simple classes. Based on this, we derive Newton's canonical forms of real cubic curves via a purely algebraic yet elementary approach.

\subsection{The Centers of Real Binary Cubic Polynomials}

We first investigate the centers of real binary cubic polynomials. A general nonzero binary cubic polynomial can be written as  
\[
f(x, y) = Ax^{3} + 3Bx^{2}y + 3Cxy^{2} + Dy^{3} + 3Ex^{2} + 6Fxy + 3Gy^{2} + 3Hx + 3Iy + K,
\]  
where \(A,B,C,D,E,F,G,H,I,K \in \mathbbm{R}\). Denote by  
\[
f_3 = Ax^{3} + 3Bx^{2}y + 3Cxy^{2} + Dy^{3}
\]  
the cubic homogeneous component of \(f\). Throughout, to avoid trivial situation, we assume $\frac{\partial f}{\partial x}$ and $\frac{\partial f}{\partial y}$ are linearly independent and \(f_3 \neq 0\). Without loss of generality, we may also assume that \(A, B, C\) are not all zero (interchanging \(x\) and \(y\) if necessary). To facilitate later discussions, we consider both \(Z(f)\) and \(Z(f_3)\) simultaneously.

For convenience, we introduce the notations  

\[
D_{4} = \begin{vmatrix}
B & C \\
F & G
\end{vmatrix}, \quad
D_{5} = \begin{vmatrix}
C & A \\
G & E
\end{vmatrix}, \quad
D_{6} = \begin{vmatrix}
A & B \\
E & F
\end{vmatrix}.
\]  
Furthermore, let  
\[
\Delta_{2} = D_5^2 - 4D_4D_6, \quad \Phi = E D_1 + F D_2 + G D_3.
\]
Similarly, $\Delta_2$ and $\Phi$ are also invariants.

\begin{theorem}\label{theo.4.2}
Keep the aforementioned notations and conventions. Every nonzero binary cubic polynomial \( f \) falls into exactly one of the following nine cases.

\begin{enumerate}
\item[(1)] \( Z(f) = Z(f_3) \cong T \) if and only if \( D_1 = D_2 = D_3 = D_4 = D_5 = D_6 = 0 \).

\item[(2)] \( \mathbbm{R}^{2} \cong Z(f) \subset Z(f_3) \cong T \) if and only if \( D_1 = D_2 = D_3 = 0 \) and \( \Delta_{2} > 0 \).

\item[(3)] \( \mathbbm{R}[\epsilon]/\langle \epsilon^{2} \rangle \cong Z(f) \subset Z(f_3) \cong T \) if and only if \( D_1 = D_2 = D_3 = 0 \), \((D_4,\ D_5,\ D_6) \ne 0\) and \( \Delta_{2} = 0 \).

\item[(4)] \( Z(f) = Z(f_3) \cong \mathbbm{R}^2 \) if and only if \( \Delta_{1} > 0 \) and \( \Phi = 0 \).

\item[(5)] \( \mathbbm{R} \cong Z(f) \subset Z(f_3) \cong \mathbbm{R}^2 \) if and only if \( \Delta_{1} > 0 \) and \( \Phi \neq 0 \).

\item[(6)] \( Z(f) = Z(f_3) \cong \mathbbm{R}[\epsilon]/\langle \epsilon^{2} \rangle \) if and only if \((D_1,\ D_2,\ D_3) \ne 0\), \( \Delta_{1} = 0 \) and \( \Phi = 0 \).

\item[(7)] \( \mathbbm{R} \cong Z(f) \subset Z(f_3) \cong \mathbbm{R}[\epsilon]/\langle \epsilon^{2} \rangle \) if and only if \( \Delta_{1} = 0 \) and \( \Phi \neq 0 \).

\item[(8)] \( Z(f) = Z(f_3) \cong \mathbbm{C} \) if and only if \( \Delta_{1} < 0 \) and \( \Phi = 0 \).

\item[(9)] \( \mathbbm{R} \cong Z(f) \subset Z(f_3) \cong \mathbbm{C} \) if and only if \( \Delta_{1} < 0 \) and \( \Phi \neq 0 \).
\end{enumerate}

\begin{proof}
According to Equation~\eqref{center}, each \( (z_{ij})_{1 \le i,j \le 2} \in Z(f) \) satisfies the following linear equations
\begin{equation}
\begin{cases}\label{5.6}
A z_{12} \;+\; B (z_{22} - z_{11}) \;-\; C z_{21} = 0, \\[2pt]
B z_{12} \;+\; C (z_{22} - z_{11}) \;-\; D z_{21} = 0, \\[2pt]
E z_{12} \;+\; F (z_{22} - z_{11}) \;-\; G z_{21} = 0.
\end{cases}
\end{equation}
Clearly, \(\dim Z(f)\) can be \(3\), \(2\), or \(1\). As \(Z(f_3)\) is the solution space of the first two equations, hence \(Z(f) \subseteq Z(f_3)\).

If \(\dim Z(f) = 3\), then Theorem~\ref{cubicform} implies \(Z(f) = Z(f_3) \cong T\). Moreover, \(\dim Z(f) = \dim Z(f_3) = 3\) holds if and only if  
\(D_1 = D_2 = D_3 = D_4 = D_5 = D_6 = 0\). This is case (1).

If \(\dim Z(f) = 2\) while \(\dim Z(f_3) = 3\), then the nonzero vectors \((A, B, C)\) and \((B, C, D)\) are linearly dependent, but \((E, F, G)\) is linearly independent of the aforementioned two vectors. Following a computation similar to that in Theorem~\ref{cubicform}, the solution of system \eqref{5.6} can be written as
\[
(z_{12},\, z_{22} - z_{11},\, -z_{21}) = b_1 (A, B, C) \times (E, F, G) = b_1 (D_4, D_5, D_6), \quad \forall b_1 \in \mathbbm{R}.
\]
Furthermore, \(\dim Z(f_3) = 3\) implies \(Z(f_3) \cong T\) by Theorem~\ref{cubicform}. Consequently, every matrix in \(Z(f) \subset Z(f_3)\) has real eigenvalues. Analogous to the proof of Theorem~\ref{cubicform}, we obtain
\[
Z(f) \cong \mathbbm{R}^{2} \;\text{ if and only if }\; \Delta_{2} > 0,
\]
and
\[
Z(f) \cong \mathbbm{R}[\epsilon]/\langle \epsilon^{2} \rangle \;\text{ if and only if }\; \Delta_{2} = 0.
\]
These are cases (2) and (3).

If \(\dim Z(f) = \dim Z(f_3) = 2\), then clearly \(\Phi = 0\). Combined with Theorem~\ref{cubicform}, we have cases (4), (6) and (8).

When \(\dim Z(f) = 1\), the three equations in (5.6) are linearly independent, which is equivalent to \(\Phi \neq 0\). It also implies that \( \dim Z(f_3) = 2\). In this situation, \(Z(f)\) consists of scalar matrices only, while \(Z(f_3)\) can be any one of the three two‑dimensional cases described in Theorem~\ref{cubicform}. These are precisely cases (5), (7) and (9).
\end{proof}
\end{theorem}

\subsection{Deriving Newton's Canonical Forms via Centers}

Based on the algebraic structures of the centers of real binary cubic polynomials and their cubic components, now we derive Newton's canonical forms of real cubic curves.

Two polynomials \(f,\ g\in\mathbbm{R}[x,\ y]\) are said to be equivalent, denoted as \(f \sim g\), if there exist \(P\in\operatorname{GL}(2,\ \mathbbm{R}),
v\in \mathbbm{R}^2\) and \(\lambda\in\mathbbm{R}\setminus\{0\}\) such that \(g(u)=\lambda f(Pu+v)\) where $u=(x,\ y)^T$.
\begin{theorem}\label{ND}
Every nonzero real binary cubic polynomial $f$ is equivalent to exactly one of the following nine canonical forms.

\begin{itemize}
    \item[(1)] \( f \sim a x^{3} + b x^{2} + c x - y + d \ \ (a \ne 0) \)  
          if and only if \( Z(f) = Z(f_3) \cong T \).

    \item[(2)] \( f \sim a x^{3} + b x^{2} - y^{2} + c x + d \ \ (a \ne 0) \)  
          if and only if \( \mathbbm{R}^{2} \cong Z(f) \subset Z(f_3) \cong T \).

    \item[(3)] \( f \sim a x^{3} + b x^{2} - x y + c x + d \ \ (a \ne 0) \)  
          if and only if \( \mathbbm{R}[\epsilon] / \langle \epsilon^{2} \rangle \cong Z(f) \subset Z(f_3) \cong T \).

    \item[(4)] \( f \sim -x y^{2} + a x^{3} + c x - e y + d \ \ (a < 0) \)  
          if and only if \( Z(f) = Z(f_3) \cong \mathbbm{R}^{2} \).

    \item[(5)] \( f \sim -x y^{2} + a x^{3} + b x^{2} + c x - e y + d \ \ (a < 0,\; b \ne 0) \)  
          if and only if \( \mathbbm{R} \cong Z(f) \subset Z(f_3) \cong \mathbbm{R}^{2} \).

    \item[(6)] \( f \sim -x y^{2} + c x - e y + d \)  
          if and only if \( Z(f) = Z(f_3) \cong \mathbbm{R}[\epsilon] / \langle \epsilon^{2} \rangle \).

    \item[(7)] \( f \sim -x y^{2} + b x^{2} + c x - e y + d \ \ (b \ne 0) \)  
          if and only if \( \mathbbm{R} \cong Z(f) \subset Z(f_3) \cong \mathbbm{R}[\epsilon] / \langle \epsilon^{2} \rangle \).

    \item[(8)] \( f \sim -x y^{2} + a x^{3} + c x - e y + d \ \ (a > 0) \)  
          if and only if \( Z(f) = Z(f_3) \cong \mathbbm{C} \).

    \item[(9)] \( f \sim -x y^{2} + a x^{3} + b x^{2} + c x - e y + d \ \ (a > 0,\; b \ne 0) \)  
          if and only if \( \mathbbm{R} \cong Z(f) \subset Z(f_3) \cong \mathbbm{C} \).
\end{itemize}
\end{theorem}

\begin{proof}
    First, by Theorems \ref{center decompose} and \ref{cubicform}, the "$\Rightarrow$" implications for each item of the theorem can be verified by direct computation. We prove only the "$\Leftarrow$" implications.

    \begin{itemize}
        \item[(1)] When \( Z(f) = Z(f_3) \cong T \), there exists an invertible matrix \( P \in \operatorname{GL}(2,\mathbbm{R}) \) such that
              \[
              P^{-1}Z(f)P = P^{-1}Z(f_3)P = T.
              \]
              After a suitable change of variables, we may assume that
              \[
              \begin{pmatrix} 0&0\\0&1 \end{pmatrix},\;
              \begin{pmatrix} 1&0\\0&0 \end{pmatrix},\;
              \begin{pmatrix} 0&0\\1&0 \end{pmatrix} \in Z(f) = Z(f_3).
              \]
              Substituting these into system \eqref{5.6} gives \( B = C = D = F = G = 0 \), hence
              \[
              f = A x^{3} + 3E x^{2} + 3H x + 3I y + K \quad (A \neq 0).
              \]
              Consequently, \( f \sim a x^{3} + b x^{2} + c x - y + d \ \ (a \neq 0) \).

        \item[(2)] When \( \mathbbm{R}^{2} \cong Z(f) \subset Z(f_3) \cong T \), there exists a \( P \in \operatorname{GL}(2,\mathbbm{R})\) such that
              \[
              P^{-1}Z(f_3)P = T, \quad
              P^{-1}Z(f)P = \left\{ \begin{pmatrix} s & 0 \\ 0 & t \end{pmatrix} \;\big|\; s,t \in \mathbbm{R} \right\}.
              \]
              After a suitable change of variables, we may assume that
              \[ 
              \begin{pmatrix} 1&0\\0&0 \end{pmatrix},\; \begin{pmatrix} 0&0\\0&1 \end{pmatrix} \in Z(f) \subset Z(f_3) \ni \begin{pmatrix} 1&0\\0&0 \end{pmatrix},\; \begin{pmatrix} 0&0\\0&1 \end{pmatrix},\; \begin{pmatrix} 0&0\\1&0 \end{pmatrix}.
              \]
              Substituting into \eqref{H3} and \eqref{5.6} yields \( B = C = D = F = 0 \), so
             \begin{align*}
					f&=Ax^{3}+3Ex^{2}+3Gy^{2}+3Hx+3Iy+K \quad (A\ne 0,\ G\ne 0)\\
					&=Ax^{3}+3Ex^{2}+(3Gy+3I)y+3Hx+K.
				\end{align*}
              After the affine transformation \( x = x',\ y = y' - \frac{I}{2G} \), we obtain
              \begin{align*}
              f=Ax'^{3}+3Ex'^{2}+3Gy'^{2}+3Hx'+K-\frac{3I^{2}}{4G} \quad (A\ne 0),
              \end{align*}
              which shows \( f \sim a x^{3} + b x^{2} - y^{2} + c x + d \ \ (a \neq 0) \).

        \item[(3)] When \( \mathbbm{R}[\epsilon]/\langle \epsilon^{2} \rangle \cong Z(f) \subset Z(f_3) \cong T \), there exists a \( P \in \operatorname{GL}(2,\mathbbm{R})\) such that
              \[
              P^{-1}Z(f_3)P = T, \quad
              P^{-1}Z(f)P = \left\{ \begin{pmatrix} s & 0 \\ t & s \end{pmatrix} \;\big|\; s,t \in \mathbbm{R} \right\}.
              \]
              After a suitable change of variables, we may assume that
              \[
               \begin{pmatrix} 0&0\\1&0 \end{pmatrix} \in Z(f) \subset Z(f_3) \ni \begin{pmatrix} 1&0\\0&0 \end{pmatrix},\; \begin{pmatrix} 0&0\\0&1 \end{pmatrix},\; \begin{pmatrix} 0&0\\1&0 \end{pmatrix}.
              \]
              Substituting into \eqref{5.6} gives \( B = C = D = G = 0 \), so
              \begin{align*}
					f&=Ax^{3}+3Ex^{2}+6Fxy+3Hx+3Iy+K \quad (A\ne 0,\ F\ne 0)\\
					&=Ax^{3}+3Ex^{2}+(6Fx+3I)y+3Hx+K.
				\end{align*}
              Applying the affine transformation \( x = x' - \frac{I}{2F},\ y = y' \) yields \(f \sim a x^{3} + b x^{2} - xy + c x + d \ \ (a \neq 0) \).
                          
        \item[(4)] When \( Z(f) = Z(f_3) \cong \mathbbm{R}^{2} \), there exists a \( P \in \operatorname{GL}(2,\mathbbm{R})\) such that
              \[
              P^{-1}Z(f)P = P^{-1}Z(f_3)P = \left\{ \begin{pmatrix} s & 0 \\ 0 & t \end{pmatrix} \;\big|\; s,t \in \mathbbm{R} \right\}.
              \]
              After a suitable change of variables, we may assume that
              \[
              \begin{pmatrix} 1&0\\0&0 \end{pmatrix},\; \begin{pmatrix} 0&0\\0&1 \end{pmatrix} \in Z(f)=Z(f_3).
              \]
              Substituting into \eqref{5.6} gives \( B = C = F = 0 \), hence
             \begin{align*}
							f&=Ax^{3}+Dy^{3}+3Ex^{2}+3Gy^{2}+3Hx+3Iy+K\\
							&=(Ax+3E)x^{2}+(Dy+3G)y^{2}+3Hx+3Iy+K.
						\end{align*}
              Applying the affine transformation \( x = x' - \frac{E}{A},\ y = y' - \frac{G}{D} \) yields
              \[
              f= A x'^{3} + D y'^{3} + H_1 x' + I_1 y' + K_1.
              \]
              By Theorem \ref{cubicform}, \( f_3 \sim -xy^{2} + a x^{3} \ \ (a < 0) \), so
              \(f\sim -xy^{2} + a x^{3} + c x - e y + d \ \ (a < 0) \).

      \item[(5)]   
When $\mathbbm{R} \cong Z(f) \subset Z(f_3) \cong \mathbbm{R}^2$, then by Theorem \ref{cubicform} we have $f_3 \sim x^3 + 3xy^2$. Therefore, the polynomial can be reduced to
\begin{align*}
f &= x^3 + 3xy^2 + E_1x^2 + F_1xy + G_1y^2 + H_1x + I_1y + K_1 \\
  &= x^3 + (3x + G_1)y^2 + E_1x^2 + F_1xy + H_1x + I_1y + K_1.
\end{align*}
Now apply the affine transformation $x = x' - \dfrac{G_1}{3}$ to obtain
\[
f = x'^3 + x'(3y^2 + F_1y) + E_2x'^2 + H_2x' + I_2y + K_2.
\]
From this expression it is clear that, after the further affine transformation $y = y' - \dfrac{F_1}{6}$, we get
\[
f = x'^3 + 3x'y'^2 + E_3x'^2 + H_3x' + I_3y' + K_3.
\]
Thus, we conclude that
\[
f \sim -xy^2 + ax^3 + bx^2 + cx - ey + d \quad (a < 0, \, b \neq 0).
\]
Note that, here $a < 0$ is due to the assumption that $Z(f_3) \cong \mathbbm{R}^2$ and $b \neq 0$ due to $Z(f) \cong \mathbbm{R}$.
       \item[(6)] 
When $Z(f_3) = Z(f) \cong \mathbbm{R}[\epsilon] / \langle \epsilon^{2} \rangle$, there exists an invertible matrix $P \in \operatorname{GL}(2,\mathbbm{R})$ such that
\[
P^{-1}Z(f_3)P = P^{-1}Z(f)P = \left\{ \begin{pmatrix} s & 0 \\ t & s \end{pmatrix} \;\middle|\; s,t \in \mathbbm{R} \right\}.
\]
After a suitable change of variables, we may assume without loss of generality that 
\[
\begin{pmatrix} 0 & 0 \\ 1 & 0 \end{pmatrix} \in Z(f), Z(f_3).
\]
Substituting this into system \eqref{5.6} yields $C = D = G = 0$. Consequently,
\begin{align*}
f(x,y) &= Ax^{3} + 3Bx^{2}y + 3Ex^{2} + 6Fxy + 3Hx + 3Iy + K \nonumber \\
&= (Ax + 3By)x^{2} + 3Ex^{2} + 6Fxy + 3Hx + 3Iy + K. 
\end{align*}
Let \( x' = Ax + 3By \) and \( y' = x \), then
\begin{align*}
f &= x'y'^2 + E_1y'^2 + F_1x'y' + H_1x' + I_1y' + K_1 \\
  &= (x' + E_1)y'^2 + F_1x'y' + H_1x' + I_1y' + K_1.
\end{align*}
Similar to the proof in item (5), after applying the affine transformation $ x' = x''-E_1 ,\ \ y' = y''-\dfrac{F_1}{2} $, we obtain
\[
f \sim -xy^2 + cx - ey + d.
\]
\item[(7)]
When \( \mathbbm{R} \cong Z(f) \subset Z(f_3) \cong \mathbbm{R}[c] / \langle c^2 \rangle \), then by Theorem \ref{cubicform} we have \( f_3 \sim xy^2 \). Consequently, the polynomial \( f \) can be reduced to
\begin{align*}
f &= -xy^2 + E_1x^2 + F_1xy + G_1y^2 + H_1x + I_1y + K_1 \\
  &= (-x + G_1)y^2 + E_1x^2 + F_1xy + H_1x + I_1y + K_1.
\end{align*}
Following a similar approach to the proof of item (5),  by applying the affine transformation $x\ = x' \  + \  G_1$,\quad $y = y' + \dfrac{F_1}{2}$, we obtain
\[
f \sim -xy^2 + bx^2 + cx - ey + d \quad (b \neq 0).
\]
The reason for \( b \neq 0 \) is the same as (5).
        \item[(8)]  
When \( Z(f_3) = Z(f) \cong \mathbbm{C} \), there exists a \( P \in \operatorname{GL}(2, \mathbbm{R}) \) such that
\[
P^{-1}Z(f_3)P = P^{-1}Z(f)P = 
\left\{ \begin{pmatrix} s & t \\ -t & s \end{pmatrix} \;\middle|\; s,t \in \mathbbm{R} \right\}.
\]
After a suitable change of variables, we may assume that
\[
\begin{pmatrix}
0 & 1 \\
-1 & 0
\end{pmatrix}
\in Z(f) = Z(f_3).
\]
Substituting this into the system of equations (3.1) yields \( A = -C ,\  B = -D \) and \( E = -G \). Thus, we have
\[
f = Ax^3 + 3Bx^2y - 3Axy^2 - By^3 + 3Ex^2 + 6Fxy - 3Ey^2 + 3Hx + 3Iy + K.
\]
According to Theorem \ref{cubicform}, after an appropriate transformation, the polynomial \( f \) can be reduced to
\[
f = x'^3 - 3x'y'^2 + E_1x'^2 + F_1x'y' - E_1y'^2 + H_1x' + I_1y' + K_1
\]
\[
= (x' + E_1)x'^2 - (3x' + E_1)y'^2 + F_1x'y' + H_1x' + I_1y' + K_1.
\]
As before, after applying the affine transformation \( x' = x'' - \dfrac{E_1}{3},\ y' = y'' + \dfrac{F_1}{6} \), we obtain
\[
f \sim -xy^2 + ax^3 + cx - ey + d \quad (a > 0).
\]

      \item[(9)] 
When $\mathbbm{R} \cong Z(f) \subset Z(f_3) \cong \mathbbm{C}$, there exists a $P \in \operatorname{GL}(2,\mathbbm{R})$ such that  
$$ P^{-1}Z(f_3)P = \left\{ \begin{pmatrix} s & t \\ -t & s \end{pmatrix} \;\middle|\; s,t \in \mathbbm{R} \right\}.$$  
After a suitable change of variables, we may assume that  
$$\begin{pmatrix} 0 & 1 \\ -1 & 0 \end{pmatrix} \in Z(f_3).$$  
By Theorem \ref{cubicform}, we have $f_3 \sim x^{3}-3xy^{2}$. Hence, the polynomial $f$ can be reduced to  
\begin{align*}
f &= x^{3}-3xy^{2}+E_{1}x^{2}+F_{1}xy+G_{1}y^{2}+H_{1}x+I_{1}y+K_{1}\\
&= x^{3}+(-3x+G_{1})y^{2}+E_{1}x^{2}+F_{1}xy+H_{1}x+I_{1}y+K_{1}.
\end{align*}
It can be readily shown that after applying the affine transformation  
$x = x' + \frac{G_{1}}{3}$,\ $y = y' + \frac{F_{1}}{6}$, we obtain  
$$f \;\sim\; -xy^{2} + ax^{3} + bx^{2} + cx - ey + d \ \ (a > 0,\; b \ne 0).$$  
\end{itemize}

    This completes the proof of the theorem.
\end{proof}

As a direct consequence of Theorem \ref{ND}, we derive Newton’s canonical forms of cubic curves using simple and natural algebraic properties of polynomials. In the following, for the convenience of the exposition, let $C$ denote a cubic curve and let $f_C$ be its defining polynomial, namely $C$ is the solution set of $f_C=0$. When there is no risk of confusion, we do not distinguish curves and their defining polynomials.

\begin{corollary}\label{CANONICAL}
Every cubic curve falls into exactly one of the following four canonical forms.
\begin{itemize}
    \item[(1)] \( C \) is of Newton's canonical form (1.4) if and only if \( Z(f_C) = Z({f_C}_3) \cong T \).
    \item[(2)] \( C \) is of Newton's canonical form (1.3) if and only if \( \mathbbm{R}^{2} \cong Z(f_C) \subset Z({f_C}_3) \cong T \).
    \item[(3)] \( C \) is of Newton's canonical form (1.2) if and only if \( \mathbbm{R}[\epsilon] / \langle \epsilon^{2} \rangle \cong Z(f_C) \subset Z({f_C}_3) \cong T \).
    \item[(4)] \( C \) is of Newton's canonical form (1.1) if and only if \( \dim Z({f_C}_3) = 2 \).
\end{itemize}
\end{corollary}

\begin{remark}
According to Theorem \ref{cubicform}, \( \dim Z(f_3) = 2 \) if and only if at least one of the quantities \( D_1, D_2, D_3 \) is nonzero. Consequently, the binary cubic polynomials satisfying \( \dim Z(f_3) = 2 \)  are the union of the solution sets of the three inequalities \( D_i \neq 0 \). In the affine space of binary cubic polynomials, those equivalent to Newton's canonical form (1.1) are thus a union of three principal open sets, each of which is evidently nonempty. Hence, they constitute a dense open subset of the affine space. In other words, almost all cubic curves can be reduced to Newton's canonical form (1.1) except a subset of measure zero.
\end{remark}

Finally in this subsection, we elucidate the approach via centers by two concrete examples of real binary cubic polynomials, showing how to transform them into the nine canonical forms described in Theorem \ref{ND} and thereby determine their Newton's canonical forms.

\begin{example}\label{new added label1}
{\upshape Consider the real binary cubic polynomial  
\[
f=8x^{3}+36x^{2}y+54xy^{2}+27y^{3}+8x^{2}+24xy+13y^{2}+6x+5y+7.
\]  

Its Hessian matrix is  
\[
H_{f}= 2\begin{pmatrix}
24x+36y+8 & 36x+54y+12 \\[2pt]
36x+54y+12 & 54x+81y+13
\end{pmatrix}.
\]  
Using Equation~\eqref{center}, the center \(Z(f_3)\) of the cubic homogeneous part is the solution space of   
\[
4z_{12} + 6z_{22} = 6z_{11} + 9z_{21}.
\]  
Hence  
\[
Z(f_{3}) = \left\{ 
\begin{pmatrix}
\frac{6a+4b-9c}{6} & b \\[4pt]
c & a
\end{pmatrix}
\;\Bigg|\; a,b,c \in \mathbbm{R} \right\} \cong T.
\]  
The center \(Z(f)\) of the full polynomial is determined by the system  
\[
\begin{cases}
4z_{12} + 6z_{22} = 6z_{11} + 9z_{21}, \\[2pt]
8z_{12} + 12z_{22} = 12z_{11} + 13z_{21},
\end{cases}
\]  
which yields  
\[
Z(f)= \left\{ 
\begin{pmatrix}
a+\frac{2b}{3} & b \\[4pt]
0 & a
\end{pmatrix}
\;\Bigg|\; a,b \in \mathbbm{R} \right\} \cong \mathbbm{R}^{2}.
\]  

From the algebraic structure of \(Z(f)\) we construct a pair of orthogonal idempotents  
\[
\epsilon_{1} = \begin{pmatrix}
1 & \frac{3}{2} \\[2pt]
0 & 0
\end{pmatrix},\quad
\epsilon_{2} = \begin{pmatrix}
0 & -\frac{3}{2} \\[2pt]
0 & 1
\end{pmatrix}.
\]  
Take an invertible matrix  
\[
P = \begin{pmatrix}
\frac{1}{2} & -\frac{3}{2} \\[2pt]
0 & 1
\end{pmatrix}
\]  
such that  
\[
P^{-1}\epsilon_{1}P = \begin{pmatrix} 1 & 0 \\ 0 & 0 \end{pmatrix},\quad
P^{-1}\epsilon_{2}P = \begin{pmatrix} 0 & 0 \\ 0 & 1 \end{pmatrix}.
\]  
Applying the variable substitution  
\[
x = \frac{1}{2}(m - 3n), \quad y = n,
\]  
we simplify \(f\) to  
\[
f = m^{3} + 2m^{2} - 5n^{2} + 3m - 4n + 7.
\] 

Therefore, $f$ belongs to case (2) in Theorem~\ref{ND} and is of Newton's canonical form \eqref{5.3}.}
\end{example}

\begin{example}\label{new added label2}
{\upshape Consider the real binary cubic polynomial  
\[
f=9x^{3}+21x^{2}y+33xy^{2}+28y^{3}+4x^{2}+2xy+5y^{2}+x+3y+1.
\]  

Its Hessian matrix is  
\[
H_{f}=2 \begin{pmatrix}
27x+21y+4 & 21x+33y+1 \\[2pt]
21x+33y+1 & 33x+84y+5
\end{pmatrix}.
\]  
The center \(Z(f)\) is the solution space of  
\[
\begin{cases}
9z_{12} + 7z_{22} = 7z_{11} + 11z_{21}, \\[2pt]
7z_{12} + 11z_{22} = 11z_{11} + 28z_{21}, \\[2pt]
4z_{12} + z_{22} = z_{11} + 5z_{21},
\end{cases}
\]  
which gives \(z_{11}=z_{22},\ z_{12}=z_{21}=0\), hence \(Z(f) \cong \mathbbm{R}\).  

The Hessian matrix of the cubic homogeneous part \(f_3\) is  
\[
H_{f_{3}}= \begin{pmatrix}
54x+42y & 42x+66y \\[2pt]
42x+66y & 66x+168y
\end{pmatrix}.
\]  
The center \(Z(f_3)\) is obtained from  
\[
\begin{cases}
9z_{12} + 7z_{22} = 7z_{11} + 11z_{21}, \\[2pt]
7z_{12} + 11z_{22} = 11z_{11} + 28z_{21},
\end{cases}
\]  
yielding  
\[
Z(f_{3}) = \left\{ 
\begin{pmatrix}
b-7a & -3a \\[2pt]
2a & b
\end{pmatrix}
\;\Bigg|\; a,b \in \mathbbm{R} \right\} \cong \mathbbm{R}^{2}.
\]  

From \(Z(f_3)\) we construct the orthogonal idempotents  
\[
\epsilon_{1} = \begin{pmatrix}
-\frac{1}{5} & -\frac{3}{5} \\[2pt]
\frac{2}{5} & \frac{6}{5}
\end{pmatrix},\quad
\epsilon_{2} = \begin{pmatrix}
\frac{6}{5} & \frac{3}{5} \\[2pt]
-\frac{2}{5} & -\frac{1}{5}
\end{pmatrix}.
\]  
There exists an invertible matrix  
\[
P = \begin{pmatrix}
\frac{3}{5} & -\frac{1}{5} \\[2pt]
-\frac{1}{5} & \frac{2}{5}
\end{pmatrix}
\]  
such that  
\[
P^{-1}\epsilon_{1}P = \begin{pmatrix} 0 & 0 \\ 0 & 1 \end{pmatrix},\quad
P^{-1}\epsilon_{2}P = \begin{pmatrix} 1 & 0 \\ 0 & 0 \end{pmatrix}.
\]  

Applying the substitution  
\[
x = \frac{1}{5}(3p - q), \quad y = \frac{1}{5}(-p + 2q),
\]  
we diagonalize the cubic component as  
\[
f_{3}\!\left(\tfrac{1}{5}(3p-q),\,\tfrac{1}{5}(-p+2q)\right) = p^{3} + q^{3}.
\]  
Consequently, the full polynomial becomes  
\[
f\!\left(\tfrac{1}{5}(3p-q),\,\tfrac{1}{5}(-p+2q)\right)
= p^{3}+ q^{3} + \tfrac{7}{5}p^{2} - \tfrac{6}{5}pq + \tfrac{4}{5}q^{2} + q + 1.
\]  
Let \(m = \tfrac{1}{2}(p+q),\ n = \tfrac{1}{2}(p-q)\) and denote the transformed polynomial by \(g(m,n)\). Then  
\begin{align*}
			g(m,n)&=2m^{3}+6mn^{2}+m^{2}+\frac{17}{5}n^{2}+\frac{6}{5}mn+m-n+1\\
			&=2m^{3}+(6m+\frac{17}{5})n^{2}+m^{2}+\frac{6}{5}mn+m-n+1.
		\end{align*}
After a translation we obtain
$$ g\left(m - \frac{17}{30}, n - \frac{1}{10}\right)=6mn^2+2m^{3}-\frac{12}{5}m^{2}+\frac{26}{15}m-\frac{42}{25}n+\frac{377}{675}.$$
The polynomial falls into case (5) of Theorem~\ref{ND}, and hence is of Newton's canonical form \eqref{5.1}.}
\end{example}

\subsection{Irreducible Cubic Curves}
In this subsection, we apply the centers of polynomials to prove that, by homogenization, change of variables, and dehomogenization, every irreducible binary cubic polynomial can be transformed into form (1.3). We first derive the criterion for irreducibility of Newton's canonical forms, next transform the condition for equivalence to form (1.3) via the center structure, and finally show that such a condition can always be satisfied. 


	For convenience of the exposition, we also consider the homogenization of Newton's canonical forms (1.1)-(1.4):
		\begin{align}
			\label{hg1} &ax^{3} - xy^2 + bx^{2}z + cxz^{2} -eyz^{2} + dz^3, \\
			\label{hg2} &ax^{3} + bx^{2}z -xyz + cxz^{2} + dz^{3}, \\
			\label{hg3} &ax^{3} + bx^{2}z - y^{2}z + cxz^{2} + dz^{3}, \\
			\label{hg4} &ax^{3} + bx^{2}z + cxz^{2} - yz^{2} + dz^{3}.
		\end{align}
Of course, the irreducibility of a cubic curve and its homogenization is identical. 

We start with the criterion of irreducibility.
\begin{theorem}	\label{n factorizable}
Canonical forms of \eqref{hg3} and \eqref{hg4} are always irreducible. A canonical form of \eqref{hg2} is irreducible if and only if $d \ne 0$. As for a canonical form of \eqref{hg1}, it is irreducible if and only if $(ce^{2}-d^{2})^{2}-ae^{6}\ne 0$ or $be^{4}-2cde^{2}+2d^{3}\ne 0$.
\end{theorem}        
			
	\begin{proof}
    We prove the contrapositive. Let $f$ be a reducible real ternary cubic form. Certainly it possesses a linear factor, denoted by $l$. Thus we may write $f=lg$, where $g$ is a ternary quadratic form.
    \begin{itemize}
    \item[(1)] If $f$ is of the form \eqref{hg3} , then $f(x, y, 0)=ax^{3}$. Therefore, the coefficient of $x$ in $l$ and that of $x^2$ in $g$ are both nonzero. Moreover, the coefficient of $y$ in $l$ and those of $xy$ and $y^{2}$ in $g$ are all $0$. Since the coefficient of $xyz$ in $f$ equals $0$, the coefficient of $yz$ in $g$ equals $0$. However, now $f=lg$ has no terms containing $y$, which leads to a conflict.
	
	\item[(2)] For forms of type \eqref{hg4}, the argument is the same as above.
	
	\item[(3)] If $f$ is of the form \eqref{hg2} and $d\ne 0$, then $f(0, y, z)=dz^{3}$ and $f(x, y, 0)=ax^{3}$. Therefore, the coefficients of $x$ and $z$ in $l$ and those of $x^2$ and $z^2$ in $g$ are all nonzero. So the coefficient of $y$ in $l$ and those of $xy$, $yz$ and $y^{2}$ in $g$ are all $0$. However, now $f=lg$ has no terms containing $y$, which leads to a conflict. On the other hand, $\eqref{hg2}$ always has a linear factor $x$ when $d=0$.

	\item[(4)] If $f$ is of the form \eqref{hg1}, then $f(0, y, z)=z^{2}(dz-ey)$. If $d=e=0$, then $f$ always has a linear factor $x$. Suppose that at least one of $d$ and $e$ is nonzero, then coefficient of $y^{2}$ in $g$ equals $0$. Since the coefficient of $xy^{2}$ in $f$ is nonzero, the coefficient of $y$ in $l$ is nonzero. Thus $e\ne 0$,  and the coefficient of $yz$ in $g$ equals $0$. Write $l=px+dz-ey$ and $g=qx^{2}+z^{2}+rxy+sxz$, and then
\[
\begin{cases}
				pq=a, \\[2pt]
				er=1,\\[2pt]
				pr-eq=0,\\[2pt]
				dr-es=0,\\[2pt]
				dq+ps=b,\\[2pt]
				ds+p=c.
\end{cases}
\]
	
	The first three equations are equivalent to $r=\frac{1}{e}$, $q=\frac{p}{e^{2}}$ and $p^{2}=ae^{2}$. Therefore $a=\frac{p^2}{e^2}\geqslant 0$. Combining these with the fourth equation, we obtain $s=\frac{d}{e^{2}}$. Substituting these into the last two equations, we have $p=c-\frac{d^{2}}{e^{2}}$ and $2dp=be^{2}$. Finally after $p$ is eliminated, we obtain $(ce^{2}-d^{2})^{2}-ae^{6}=be^{4}-2cde^{2}+2d^{3}=0$. Note that $d=e=0$ also implies the two equations. Now assume that
 \[(ce^{2}-d^{2})^{2}-ae^{6}=be^{4}-2cde^{2}+2d^{3}=0\] in \eqref{hg1}. If $e=0$, then $d=0$ and $\eqref{hg1}$ has a linear factor $x$. If $e\ne 0$ and $d=0$, then $b=0$ and $a=\frac{c^{2}}{e^{2}}$. Let $l=\frac{c}{e}x-y$ and $g=\frac{c}{e}x^{2}+ez^{2}+xy$. Then \eqref{hg1} factors into $lg$. If neither $d$ nor $e$ equals $0$, then \eqref{hg1} factors into $lg$ where $l=\frac{be^{2}}{2d}x+dz-ey$ and $g=\frac{b}{2d}x^{2}+z^{2}+\frac{1}{e}xy+\frac{d}{e^{2}}xz$.
\end{itemize}
	Now the proof is completed.
\end{proof}

\begin{remark} 
	According to the proof, we have two simpler sufficient conditions for the irreducibility of \eqref{5.1} which are $e\ne 0, a<0$ and $e=0, d\ne 0$. Furthermore, during the process of eliminating $p$, an extra expression $b^{2}e^{2}-4ad^{2}$ occurs. Therefore $b^{2}e^{2}-4ad^{2}\ne 0$ is another sufficient condition. Besides, the two inequalities in the equivalence are independent. For example, when $(a, b, c, d, e)=(1, 0, 0, 0, 1)$, then $(ce^{2}-d^{2})^{2}-ae^{6}=-1\ne 0=be^{4}-2cde^{2}+2d^{3}$, while $(a, b, c, d, e)=(1, 0, 0, 1, 1)$, then $(ce^{2}-d^{2})^{2}-ae^{6}=0\ne 2=be^{4}-2cde^{2}+2d^{3}$.
\end{remark}

\begin{example}
	The polynomial in Example \ref{new added label1} is of the form \eqref{5.3}, which is irreducible. The polynomial in Example \ref{new added label2} is of the form \eqref{5.1}. The corresponding coefficients satisfy $e\ne 0$ and $a<0$. Thus it is irreducible.
\end{example}


	

Next we will apply the centers to prove that each irreducible binary cubic polynomial can be transformed into the following form
\begin{equation}
y^2=x^3+cx+d \label{nselp}
\end{equation} 
which is a translation of the form \eqref{5.3}. Let $J_{f}=(f_{x}, f_{y}, f_{z})$ denote the Jacobian matrix of the form $f(x, y, z)$. 
\begin{theorem}	\label{WSF and FLEX0}
Let $u=(x, y, z)^T$ and let $f(u)$ be an irreducible real ternary cubic form. Then $f\sim x^{3}+cxz^{2}+dz^{3}-y^{2}z$ if and only if there exist non-proportional $p_{1}$ and $p_{2}\in\mathbbm{R}^{3}$ such that $f(p_{2})=0$, $p_{1}^T H_f(p_{2})=0$ and $J_{f}(p_{2})\ne 0$.
	\end{theorem}

	\begin{proof}
	For any irreducible real ternary cubic form $f(u)$ and invertible matrix $P=(p_{1}, p_{2}, p_{3})\in\operatorname{GL}(3, \mathbbm{R})$ where $p_{1}$, $p_{2}$, $p_{3}\in\mathbbm{R}^3$ are linearly independent, $g(u)=f(Pu)$ is also irreducible. Hence the dehomogenization with respect to $z$ of $g$, denoted by $g_{*}(x, y)$, is always a binary cubic polynomial. As $g_{*}(x, y)=g(x, y, 1)$, ${g_{*}}_3(x, y)=g(x, y, 0)$ and $H_g(u)=P^{T} H_f(Pu)P$, it follows that
\begin{align*}
	H_{g_{*}}= \begin{pmatrix}
		g_{xx}(x, y, 1) & g_{xy}(x, y, 1) \\[2pt]
		g_{xy}(x, y, 1) & g_{yy}(x, y, 1)
	\end{pmatrix}
	= \begin{pmatrix}
		p_{1}^T H_f(P(x, y, 1)^T) p_{1} & p_{1}^T H_f(P(x, y, 1)^T) p_{2} \\[2pt]
		p_{2}^T H_f(P(x, y, 1)^T) p_{1} & p_{2}^T H_f(P(x, y, 1)^T) p_{2} 
\end{pmatrix},
	\\[2pt]
	H_{{g_{*}}_3}=	\begin{pmatrix}
		g_{xx}(x, y, 0) & g_{xy}(x, y, 0) \\[2pt]
		g_{xy}(x, y, 0) & g_{yy}(x, y, 0)
\end{pmatrix}
	= \begin{pmatrix}
		p_{1}^T H_f(P(x, y, 0)^T) p_{1} & p_{1}^T H_f(P(x, y, 0)^T) p_{2} \\[2pt]
		p_{2}^T H_f(P(x, y, 0)^T) p_{1} & p_{2}^T H_f(P(x, y, 0)^T) p_{2} 
	\end{pmatrix}.
\end{align*} 

	By completing powers, case (2) of Theorem \ref{ND} is always equivalent to $x^{3}+cx+d-y^{2}$. So $f\sim x^{3}+cxz^{2}+dz^{3}-y^{2}z$ if and only if there exists $P\in\operatorname{GL}(3, \mathbbm{R})$ such that $g(u)=f(Pu)$ satisfies
\[
\left\{ 
\begin{pmatrix}
s & 0 \\[2pt]
0 & t
\end{pmatrix}
\;\Bigg|\; s, t \in \mathbbm{R} \right\} = Z(g_{*}) \subset Z({g_{*}}_3) = T,
\]
if and only if there exists $P=(p_{1}, p_{2}, p_{3})\in\operatorname{GL}(3, \mathbbm{R})$ such that	
\[
\begin{cases}
			p_{1}^T H_f\left(P(x, y, 1)^T\right) p_{2}=0, \\[2pt]
			p_{1}^T H_f\left(P(x, y, 0)^T\right) p_{2}=0, \\[2pt]
			p_{2}^T H_f\left(P(x, y, 0)^T\right) p_{2}=0, \\[2pt]
			p_{1}^T H_f\left(P(x, y, 1)^T\right) p_{1}\ne 0, \\[2pt]
			p_{2}^T H_f\left(P(x, y, 1)^T\right) p_{2}\ne 0, \\[2pt]
			p_{1}^T H_f\left(P(x, y, 0)^T\right) p_{1}\ne 0.
\end{cases}
\]  

The correspondence between cubic forms and symmetric $3$-linear functions yields that for any $v\in\mathbbm{R}^3$, the equality $p_{i}^T H_f(v) p_{j}=p_{i}^T H_f(p_{j}) v$ holds. In fact, given an arbitrary ternary cubic form $\varphi=\sum_{1\leqslant i, j, k\leqslant 3}\lambda_{ijk}x_i x_j x_k$ where $\lambda_{l_1 l_2 l_3}=\lambda_{l_{\sigma(1)} l_{\sigma(2)} l_{\sigma(3)}}$ for every $\sigma\in S_3$, three arbitrary vectors $a_i=(a_{i1}, a_{i2}, a_{i3})\in \mathbbm{R}^3 ( i=1, 2, 3)$ and an arbitrary permutation $\tau\in S_3$, we have
\begin{align*}
a_1 H\left(a_2^T\right) a_3^T&=\sum_{1\leqslant l_1, l_2, l_3\leqslant 3} \lambda_{l_1 l_2 l_3} a_{1 l_1} a_{2 l_2} a_{3 l_3}\\
&=\sum_{1\leqslant l_1, l_2, l_3\leqslant 3} \lambda_{l_{\tau^{-1}(1)} l_{\tau^{-1}(2)} l_{\tau^{-1}(3)}} a_{\tau(1) l_1} a_{\tau(2) l_2} a_{\tau(3) l_3}\\
&=\sum_{1\leqslant l_1, l_2, l_3\leqslant 3} \lambda_{l_1 l_2 l_3} a_{\tau(1) l_1} a_{\tau(2) l_2} a_{\tau(3) l_3}\\
&=a_{\tau(1)} H\left(a_{\tau(2)}^T\right) a_{\tau(3)}^T.
\end{align*}
    Meanwhile, Euler's theorem for homogeneous functions indicates that $p_{i}^T H_f(p_{i})=2J_{f}(p_{i})$ and $J_{f}(p_{i})p_{i}=3f(p_i)$. Note that $x$, $y$, $1$ are $\mathbbm{R}$-linearly independent, the above system of equations and inequalities is equivalent to
\[
\begin{cases}
			p_{1}^T H_f(p_{2})=0, \\[2pt]
			f(p_{2})=0, \\[2pt]
			f(p_{1})\ne 0, \\[2pt]
			J_{f}(p_{2})\ne 0.
\end{cases}
\]  
For fixed $p_1$ and $p_2$ which are linearly independent, any choice of $p_3$ linearly independent with them, $p_1\times p_2$ for example, makes $P$ invertible.

	Now it suffices to show that the third inequality is redundant under the condition that $f$ is irreducible. In fact, if $f(p_{1})=f(p_{2})=0$ and $p_{1}^T H_f(p_{2})=0$, then $g(0, 1, 0)=g(1, 0, 0)=0$ and $g_{xy}=0$. Thus $g$ is divisible by $z$, which contradicts the irreducibility.
\end{proof}
 
	Denote the determinant of the Hessian matrix of $f$, i.e., the Hessian determinant $\det H(f)$, by $h_f$. Then $[p_{2}]\in\mathbbm{R}P^{2}$ in the above theorem satisfies $h_{f}(p_{2})=f(p_{2})=0$ and $J_{f}(p_{2})\ne 0$. Conversely, for any $[p_{2}]\in\mathbbm{R}P^{2}$ such that $h_{f}(p_{2})=f(p_{2})=0$ and $J_{f}(p_{2})\ne 0$, there exists $p_{1}\in\mathbbm{R}^3$ linearly independent with $p_{2}$ such that $p_{1}^T H_f(p_{2})=0$. 


\begin{remark} 	\label{algflex}
Recall that a point $[p]\in\mathbbm{R}P^{2}$ is said to be a flex of the ternary cubic $f$ if  $h_{f}(p)=f(p)=0$ and $J_{f}(p)\ne 0$. In the previous theorem, we rediscover flex points which were usually defined in the more geometric manner, see \cite{B}.
	\end{remark} 


	
Finally we are ready to prove that all irreducible cubic curves can be transformed into \eqref{nselp}.

\begin{theorem}
  Every irreducible binary cubic polynomial can be transformed into the form \eqref{nselp}.  
\end{theorem}

	\begin{proof}
		Let $f(x, y, z)$ be an irreducible ternary cubic form defining a curve $C$. By Corollary \ref{CANONICAL}, we may assume that $f$ has one of the forms \eqref{hg1}-\eqref{hg4}. Theorem \ref{WSF and FLEX0} and Remark \ref{algflex} imply that it suffices to show $C$ has a flex.
		
    \begin{itemize}
        \item[(1)] If $f$ is of the form \eqref{hg3}, that is, $f=ax^{3} + bx^{2}z - y^{2}z + cxz^{2} + dz^{3}$ where $a\ne 0$, then after a translation, we may assume $b=0$ and $f$ itself is of the form \eqref{nselp}.
        
        \item[(2)] If $f$ is of the form \eqref{hg4}, that is, $f=ax^{3} + bx^{2}z+ cxz^{2} - yz^{2} + dz^{3}$ where $a\ne 0$, then $f_y=-z^2$ and $h_f=-8z^2(3ax+bz)$. Put $y_{0}=bc-3ad-\frac{2b^3}{9a}$. Then $f(b, y_{0}, -3a)=h_f(b, y_{0}, -3a)=0$ and $f_y(b, y_{0}, -3a)=-9a^2\ne 0$. And thus, $J_f(b, y_{0}, -3a)\ne 0$. Therefore, $[b: y_0: -3a]$ is a flex of $C$.
        
        \item[(3)] If $f$ is of the form \eqref{hg2}, that is, $f=ax^{3} + bx^{2}z -xyz + cxz^{2} + dz^{3}$, then $a d\ne 0$ by Theorem \ref{n factorizable}. It follows that
$$g(x, y, z)=f\left(\frac{x}{\sqrt[3]{a}}, \sqrt[3]{ad}y+\frac{bx}{\sqrt[3]{a}}+\frac{cz}{\sqrt[3]{d}}, \frac{z}{\sqrt[3]{d}}\right)=x^{3}+z^{3}-xyz.$$
As $J_g=(3x^2-yz, -xz, 3z^2-xy)$ and $h_g=-2(3x^{3}+3z^{3}+xyz)$, we see that $g(1, 0, -1)=h_g(1, 0, -1)=0$ and $J_g(1, 0, -1)=(3, 1, 3)\ne 0$. Therefore, $[1: 0: -1]$ is a flex of $C': g=0$, which implies that there is a flex on $C$.
        
        \item[(4)] Now it remains to consider the case of $f$ being of the form \eqref{hg1}. After rescaling, we may write
$$f=ax^3+3bx^2z+3cxz^2+dz^3-3xy^2-3eyz^2.$$ 

If $e=0$, then $d\ne 0$ by the irreducibility and $f\left(\frac{1}{3}z, y, x\right)=dx^{3}+cx^{2}z+\frac{b}{3}xz^{2}+\frac{a}{27}z^{3}-y^2z$ is of the form \eqref{hg3} and $C$ has a flex.

If $e\ne 0$, then let $e=1$ without loss of generality. We split the proof into three steps.

\noindent\textit{Step 1: Reducing this case to $ax^3+3cxz^2+dz^3+3xy^2+3yz^2 \ (a>0)$.}

 Let $p$ be a real root of the equation $p^3+dp^2+2cp+b=0$ and
$$g(x, y ,z)=f(x', y', z')= f(x, z-py+(c+dp+p^2)x, y+px).$$
The irreducibility ensures that $g_{*}$ is a binary polynomial of degree 3. Since$$
		g_{xy}(x, y ,z)=6(p^3+dp^2+2cp+b)x
		=0,$$
 $H_{g_{*}}$ is a diagonal matrix. Accordingly, $Z(g_{*})\cong\mathbbm{R}^2$ or $T$. By Theorem \ref{ND}, $f$ is equivalent to either form \eqref{hg3} or \eqref{hg4}, each possessing a flex, or to the form$$ax^3+3cxz^2+dz^3+3xy^2+3yz^2 \ (a>0).$$
\noindent\textit{Step 2: Finding a flex or a singularity.}

 Write $f(x, y, z)=ax^3+3cxz^2+dz^3+3xy^2+3yz^2 \ (a>0)$. Its Hessian determinant is
$$h_f=216(acx^3+ax^2y-cxy^2-y^3+adx^2z-(a+c^2)xz^2-dy^2z+2cyz^2).$$ 
Recall that for a nonzero solution $p=(x_0,y_0,z_0)$ of $f=h_f=0$, $[p]$ is a flex of $C$ if $J_f(p)\neq 0$, and a singularity of $C$ if $J_f(p)=0$. It is easy to verify that $z_0\ne 0$ for any nonzero solution $(x_0, y_0, z_0)$ of $f=h_f=0$.

\noindent\textit{Case 1: $d(d^2+9c)=0$.}

In this case,
$$f(0, d, -3)=h_f(0, d, -3)=0,$$ so $f=h_f=0$ does have a nonzero solution. 

\noindent\textit{Case 2: $d(d^2+9c)\ne 0$.}

In this situation we have $x_0\ne 0$. Otherwise, putting $x_0=0$ into the system $f=h_f=0$ yields $d^2+9c=0$. Set $z=1$ and $y=kx$. Then the equation system $f=\frac{1}{x} h_f=0$ becomes 
 	\begin{align} \label{Re0}
		\begin{cases}
(a+3k^2)x^3+3(c+k)x+d=0,       \\[2pt]
(a-k^2)(c+k)x^2+d(a-k^2)x-(a+c^2-2ck)=0. 
		\end{cases}
	\end{align}
Since $a+3k^2>0$, the first equation is always cubic. Then direct computation shows that the leading term of the resultant of these two equations is $2(d^2+9c)k^9$. Hence there is a real $k_0$ such that the resultant vanishes and a complex solution to \eqref{Re0} exists.

If $(a-k_0^2)(c+k_0)=0$, then \eqref{Re0} must have a real solution. If $(a-k_0^2)(c+k_0) \ne 0$ and the second equation has a pair of non-real roots, then it has to divide the first real cubic equation. Euclidean division then gives $d(d^2+9c)=0$, contrary to the assumption. Therefore, it can only have real solutions and so does the whole \eqref{Re0}. Thus, there is a $[q]\in C$ such that $f(q)=h_f(q)=0$, which indicates that $[q]$ is a flex or a singularity. 

If $C$ is nonsingular, then $J(q)\ne 0$ and $q$ is a flex of $C$. 

\noindent\textit{Step 3: Coping with the singular case.}

Now suppose that $C$ is singular. The condition for a singularity is $J_f=0$, that is, the system of equations
 	\begin{align}\label{sp1}
		\begin{cases}
 	f_x=3(ax^2+y^2+cz^2)=0, \\[2pt]
	f_y=3(2xy+z^2)=0, \\[2pt]
	f_z=3(2cxz+2yz+dz^2)=0.                     
		\end{cases}
	\end{align}
As $J_f=0$ implies $f=h_f=0$, we may assume $z=1$. Then system \eqref{sp1} can be transformed to
 	\begin{align}\label{sp2}
		\begin{cases}
 	y=-cx-\frac{d}{2}, \\[2pt]
	(a+c^2)x^2+cdx+\frac{d^2}{4}+c=0, \\[2pt]
	2cx^2+dx-1=0.                     
		\end{cases}
	\end{align}
Since $a+c^2>0$, the existence of a complex solution is equivalent to the vanishing of the resultant of the last two quadratic equations of \eqref{sp2}, which gives
$$4a^2+ad^4+12acd^2+24ac^2+4c^3d^2+36c^4=0.$$
Viewing it as a quadratic equation in $a$ and by completing powers we obtain
$$(8a+d^4+12cd^2+24c^2)^2=d^2(d^2+8c)^3.$$
Therefore, $d\ne 0$ and $d^2+8c\geqslant 0$. Since the discriminant of the third equation of \eqref{sp2} is $d^2+8c$, we get the existence of a real solution. Moreover, we see from \eqref{sp1} that any solution $(x_0, y_0, 1)$ satisfies $x_0y_0=-\frac{1}{2}$ and thus $x_0$ and $y_0$ are both nonzero.

\noindent\textit{Case 1: $a=c^2$.}

Substituting this into \eqref{sp2} gives $d^2=-8c$. There is a unique solution $\left(\frac{2}{d}, -\frac{d}{4}\right)$ to \eqref{sp2}. Hence there is a unique singular point $\left[\frac{2}{d} : -\frac{d}{4} : 1\right]=[2 : 2c : d]$ on $C$. On the other hand, $[4 : 4c : -d]\ne [2 : 2c : d]$ is another solution to $f=h_f=0$. Hence it is a flex of $C$.

\noindent\textit{Case 2: $a\ne c^2$.}

Solve \eqref{sp2} and we obtain the unique singularity $[x_0 : y_0 : z_0]=[cd^2+6c^2+2a: -(ad^2+6c^3+2ac): 2d(a-c^2)]$ of $C$, where $x_0 y_0 z_0\ne 0$. Put $g(x, y, z)=f(x+x_0 z, y+y_0 z, z_0 z)$ and $C': g=0$. Then $[0 : 0 : 1]$ is the unique singular point of $C'$. It suffices to show that $C'$ has a flex. Since $Jg (0, 0, 1) = 0$, the coefficients of $z^3, xz^2, yz^2$ in $g$ are all zero. Compare the coefficients and we get
	\begin{equation*}
		g(x, y, z)=3(ax_0 x^2+2y_0 xy+x_0 y^2)z+x(ax^2+3y^2).
	\end{equation*}
Next we will solve $g=h_g=0$. After eliminating $z$ we obtain
$$G_1(x, y)=3x_0y_0y^3+3ax_0^2xy^2-3ax_0y_0x^2y-a(ax_0^2+2y_0^2)x^3=0.$$
If there is a nonzero real solution of $G_1$ which is not a solution of
$$G_2(x, y)=ax_0 x^2+2y_0 xy+ x_0y^2=0,$$
solving $g=0$ will give the expected $z$.

Since $a(ax_0^2+2y_0^2)>0, ax_0\ne 0$, $G_1$ is always a real cubic and $G_2$ is always a real quadratic. Set $y=1$. $G_1=0$ always has a nonzero real solution and it remains to show that the resultant 
$$\mathrm{Res}(G_1(x, 1), G_2(x, 1))=16ax_0(ax_0^2+3y_0^2)(ax_0^2-y_0^2)^2\ne 0.$$
Otherwise, suppose that $y_0^2=ax_0^2$, then $J_f(x_0, y_0, z_0)=0$ implies that $-2cx_0y_0=cz_0^2=-y_0^2-ax_0^2=-2ax_0^2$. Consequently, $ax_0=cy_0$. It follows that $a^2x_0^2=c^2y_0^2=ac^2 x_0^2$ and $a=c^2$, so we arrive at a contradiction. Hence $\mathrm{Res}(G_1(x, 1), G_2(x, 1))\ne 0$ and the expected real solution does exist. Therefore, $C$ has a flex.

\end{itemize}
The theorem is now proved.
\end{proof}

\section*{Use of AI Tools Declaration}

The authors declare they have not used Artificial Intelligence (AI) tools in the creation of this article.

\section*{Acknowledgments}

H.-L. Huang was partially supported by the Key Program of Natural Science Foundation of Fujian Province (Grant No. 2024J02018) and the National Natural Science Foundation of China (Grant No. 12371037). Y. Ye was partially supported by the National Key R\&D Program of China (No. 2024YFA1013802), the National Natural Science Foundation of China (Nos. 12131015 and 12371042), and the Quantum Science and Technology-National Science and Technology Major Project (No. 2021ZD0302902).

\section*{Conflict of Interest}

The authors declare no conflicts of interest.


\begin{spacing}{.88}
\setlength{\bibsep}{2.pt}

\end{spacing}
\end{document}